\documentclass{gtpart}
\usepackage{pinlabel}
\usepackage{graphicx}
\title[Poly-freeness of even Artin groups of FC type]{Poly-freeness of even Artin groups of FC type}
\author[R Blasco-Garc\'ia]{Rub\'en Blasco-Garc\'ia}
\givenname{Rub\'en}
\surname{Blasco}
\address{Departamento de Matem\'aticas\\
Universidad de Zaragoza\\
50009 Zaragoza\\
Spain}
\email{rubenb@unizar.es}
\urladdr{}
\author[C Mart\'inez-P\'erez]{Conchita Mart\'inez-P\'erez}
\givenname{Conchita}
\surname{Mart\'inez-P\'erez}
\address{Departamento de Matem\'aticas\\
Universidad de Zaragoza\\
50009 Zaragoza\\
Spain}
\email{conmar@unizar.es}
\urladdr{}
\author[L Paris]{Luis Paris}
\givenname{Luis}
\surname{Paris}
\address{IMB, UMR 5584\\
CNRS, Univ. Bourgogne Franche-Comt\'e\\
21000 Dijon\\
France}
\email{lparis@u-bourgogne.fr}
\urladdr{}


\volumenumber{}
\issuenumber{}
\publicationyear{}
\papernumber{}
\startpage{}
\endpage{}
\doi{}
\MR{}
\Zbl{}
\received{}
\revised{}
\accepted{}
\published{}
\publishedonline{}
\proposed{}
\seconded{}
\corresponding{}
\editor{}
\version{}
\newtheorem{thm}{Theorem}[section]
\newtheorem{lem}[thm]{Lemma}
\newtheorem{prop}[thm]{Proposition}

\theoremstyle{definition}
\newtheorem*{rem}{Remark}

\numberwithin{equation}{section}

\makeatletter
\renewcommand{\thefigure}{\ifnum \c@section>\z@ \thesection.\fi
 \@arabic\c@figure}
\@addtoreset{figure}{section}
\makeatother

\def\NN{\mathcal N} \def\N{\mathbb N} \def\lk{{\rm lk}}
\def\Ker{{\rm Ker}} \def\Id{{\rm Id}} \def\supp{{\rm supp}}

\begin{document}


\begin{abstract}
We prove that even Artin groups of FC type are poly-free and residually finite.
\end{abstract}

\maketitle


\section{Introduction}

One of the families of groups where the interaction between algebraic and geometric techniques has been most fruitful is the family of Artin groups, also known as Artin-Tits groups or generalized braid groups. 
There are very few works dealing with all the Artin groups, and the theory mainly consists on the study of more or less extended subfamilies. 
The subfamily of right-angled Artin groups is one of the most studied, in particular because of they resounding applications in homology of groups by Bestvina--Brady \cite{BesBra1} and in low dimensional topology by Haglund--Wise \cite{HagWis1} and Agol \cite{Agol1}.
We refer to Charney \cite{Charn1} for an introductory survey on these groups.

The  main goal of the present paper is to prove that some known structural property of right-angled Artin groups also holds for a larger family of groups (see Theorem \ref{Thm3_17}).
The property is the ``poly-freeness'' and the family is that of even Artin groups of FC type.

A group $G$ is called \emph{poly-free} if there exists a tower of subgroups 
\[
1= G_{0} \unlhd G_{1} \unlhd  \cdots \unlhd G_{N}=G 
\]
for which each quotient $G_{i+1}/G_i$ is a free group.
Recall that a group is \emph{locally indicable} if every non-trivial finitely generated subgroup admits an epimorphism onto $\Z$, and a group is \emph{right orderable} if it admits a total order invariant under right multiplication.
Poly-free groups are locally indicable, and locally indicable groups are right orderable (see Rhemtulla--Rolfsen \cite{RheRol1}).
The fact that right-angled Artin groups are poly-free was independently proved by Duchamp--Krob \cite{DucKro1}, Howie \cite{Howie1} and  Hermiller--\v{S}uni\'c \cite{HerSun1}.

Artin groups of FC type were introduced by Charney--Davis \cite{ChaDav1} in their study of the $K (\pi, 1)$ conjecture for Artin groups.
They are particularly interesting because there are CAT(0) cubical complexes on which they act. 
This family contains both, the right-angled Artin groups, and the Artin groups of spherical type (those corresponding to finite Coxeter groups). 
An Artin group is called \emph{even} if it is defined by relations of the form $(st)^k = (ts)^k$, $k \ge 1$.
There are some papers in the literature dealing with even Coxeter groups (see for example \cite{AnCio}) but as far as we know, there is no other paper dealing with even Artin groups (although they are mentioned in \cite{ArCogMat}). However, we think that even Artin groups deserve to be studied because they have remarkable properties.  
One of them is that such a group retracts onto any parabolic subgroup (see Section \ref{Sec2}).

Our proof is partially inspired by the one of Duchamp--Krob \cite{DucKro1} for right-angled Artin groups in the sense that, as them, we define a suitable semi-direct product  of an Artin group based on a proper subgraph with a free group, and we show that this semi-direct product is isomorphic to the original Artin group. 
However, in our case, this approach is more complicated, essentially because we have to deal with relations that are also more complicated.

In Section \ref{Sec4} we prove that even Artin groups of FC type are residually finite. 
This result is a more or less direct consequence of Boler--Evans \cite{BolEva1}, but it deserves to be pointed out since it increases the list of Artin groups shown to be residually finite. 
Actually, this list is quite short.
It contains the spherical type Artin groups (since they are linear), the right-angled Artin groups (since they are residually nilpotent), and few other examples.
In particular, it is not known whether all Artin groups of FC type are residually finite.

\subsection*{Acknowledgements} {We thank Yago Antol\'in for suggesting us the main problem studied in this paper.
The first two named authors were partially supported by Gobierno de Arag\'on, European Regional 
Development Funds and 
MTM2015-67781-P (MINECO/FEDER). The first named author was moreover supported by the Departamento de Industria e Innovaci\'on del Gobierno de Arag\'on and Fondo Social Europeo Phd grant}


\section{Preliminaries}\label{Sec2}

\subsection{Artin groups}

Let $S$ be a finite set. 
A \emph{Coxeter matrix} over $S$ is a square matrix $M=(m_{s,t})_{s,t\in S}$ indexed by the elements of $S$, with coefficients in $\N \cup \{ \infty\}$, and satisfying $m_{s,s}=1$ for all $s \in S$ and $m_{s,t} = m_{t,s} \ge 2$ for all $s,t \in S$, $s \neq t$.
We will represent such a Coxeter matrix $M$ by a labelled graph $\Gamma$, whose set of vertices is $S$, and where two distinct vertices $s,t \in S$ are linked by an edge labelled with $m_{s,t}$ if $m_{s,t} \neq \infty$.
We will also often use the notation $V(\Gamma)$ to denote the set of vertices of $\Gamma$ (that is, $V(\Gamma) = S$).

\begin{rem}
The labelled graph $\Gamma$ defined above is not the Coxeter graph of $M$ as defined in Bourbaki \cite{Bourb1}.
It is another fairly common way to represent a Coxeter matrix. 
\end{rem}

If $a,b$ are two letters and $m$ is an integer $\ge 2$, we set $\Pi(a,b:m) = (ab)^{\frac{m}{2}}$ if $m$ is even, and $\Pi(a,b:m) = (ab)^{\frac{m-1}{2}}a$ if $m$ is odd. 
In other words, $\Pi(a,b:m)$ denotes the word $aba \cdots$ of length $m$.
The \emph{Artin group} of $\Gamma$, $A=A_\Gamma$, is defined by the presentation
\[
A = \langle S \mid \Pi(s,t: m_{s,t}) = \Pi(t,s: m_{s,t}) \text{ for all } s,t \in S,\ s \neq t \text{ and } m_{s,t} \neq \infty \rangle\,.
\]
The \emph{Coxeter group} of $\Gamma$, $W=W_\Gamma$, is the quotient of $A$ by the relations $s^2=1$, $s \in S$.

For $T \subset S$, we denote by $A_T$ (resp. $W_T$) the subgroup of $A$ (resp. $W$) generated by $T$, and by $\Gamma_T$ the full subgraph of $\Gamma$ spanned by $T$. 
Here we mean that each edge of $\Gamma_T$ is labelled with the same number as its corresponding edge of $\Gamma$.
By Bourbaki \cite{Bourb1}, the group $W_T$ is the Coxeter group of $\Gamma_T$, and, by van der Lek \cite{Lek1}, $A_T$ is the Artin group of $\Gamma_T$.
The group $A_T$ (resp. $W_T$) is called a \emph{standard parabolic subgroup} of $A$ (resp. of $W$).

As pointed out in the introduction, the theory of Artin groups consists in the study of more or less extended families. 
The one concerned by the present paper is the family of even Artin groups of FC type. 
These are defined as follows. 
We say that $A=A_\Gamma$ is of \emph{spherical type} if its associated Coxeter group, $W_\Gamma$, is finite.
A subset $T$ of $S$ is called \emph{free of infinity} if $m_{s,t} \neq \infty$ for all $s,t \in T$.
We say that $A$ is of \emph{FC type} if $A_T$ is of spherical type for every free of infinity subset $T$ of $S$.
Finally, we say that $A$ is \emph{even} if $m_{s,t}$ is even for all distinct $s,t \in S$ (here, $\infty$ is even).

\begin{rem}
As we said in the introduction, we think that even Artin groups deserve special attention since they have particularly interesting properties as the following ones.
Assume that $A$ is even. 
\begin{itemize}
\item[(1)]
Let $s,t \in S$, $s \neq t$.
If we set $m_{s,t} = 2k_{s,t}$, then the Artin relation $\Pi(s,t:m_{s,t}) = \Pi(t,s:m_{s,t})$ becomes $(st)^{k_{s,t}}= (ts)^{k_{s,t}}$.
This form of relation is less innocuous than it seems (see Subsection \ref{SubSec2_3} for example).
\item[(2)]
Let $T$ be a subset of $S$.
Then the inclusion map $A_T \hookrightarrow A$ always admits a retraction $\pi_T : A \to A_T$ which sends $s$ to $s$ if $s \in T$, and sends $s$ to $1$ if $s \not \in T$.
\end{itemize}
\end{rem}

\subsection{Britton's lemma} 

Let $G$ be a group generated by a finite set $S$.
We denote by $(S \cup S^{-1})^*$ the free monoid over $S \cup S^{-1}$, that is, the set of words over $S \cup S^{-1}$, and we denote by $(S \cup S^{-1})^* \to G$, $w \mapsto \bar w$, the map that sends a word to the element of $G$ that it represents. 
Recall that a set of \emph{normal forms} for $G$ is a subset $\NN$ of $(S \cup S^{-1})^*$ such that the map $\NN \to G$, $w \mapsto \bar w$, is a one-to-one correspondence.

Let $G$ be a group with two subgroups $A,B\leq G$, and let $\varphi: A\rightarrow B$ be an isomorphism. 
A useful consequence of Britton's lemma yields a set of normal forms for the HNN-extension $G*_{\varphi}=\langle G,t \mid t^{-1}at=\varphi(a), a\in A\rangle$ in terms of a set $\NN$ of normal forms for $G$ and sets of representatives of the cosets of $A$ and $B$ in $G$ (see Lyndon--Schupp \cite{LynSch1}). 
Explicitly, choose a set $T_A$ of representatives of the left cosets of $A$ in $G$ containing $1$, and a set $T_B$ of representatives of the left cosets of $B$ in $G$ also containing $1$.

\begin{prop}[Britton's normal forms]\label{Prop2_1}
Let $\tilde \NN$ be the set of words of the form $w_0 t^{\varepsilon_1} w_1 \cdots t^{\varepsilon_m} w_m$, where $m\ge 0$, $\varepsilon_i \in \{\pm 1 \}$ and $w_i\in \NN$ for all $i$, such that:
\begin{itemize}
\item[(a)]
$\bar w_i \in T_A$ if $\varepsilon_i=-1$, for $i \ge 1$,
\item[(b)]
$\bar w_i \in T_B$ if $\varepsilon_i=1$, for $i \ge 1$,
\item[(c)]
there is no subword of the form $t^{\varepsilon}t^{-\varepsilon}$.
\end{itemize}
Then $\tilde\NN$ is a set of normal forms for the HNN-extension $G*_{\varphi}$.
\end{prop}

\subsection{Variations of the even Artin-type relations}\label{SubSec2_3}

For $a,b$ in a group $G$, we denote by $b^a = a^{-1} b a$ the conjugate of $b$ by $a$.
The aim of this subsection is to illustrate how the Artin relations in even Artin groups can be expressed in terms of conjugates. 
This observation will be a key point in our proof of Theorem \ref{Thm3_17}.

Let $s,t$ be two generators of an Artin group $A$.
The relation $st=ts$ can be expressed as $t^{s}=t$.
Analogously, from $tsts=stst$, we obtain $s^{-1}tst=tsts^{-1}$, and therefore $t^{s^{-1}}=t^{-1}t^{s}t$ and $t^{s^{2}}=t^{s}t(t^{s})^{-1}$. 
We deduce that $t^{s^{-1}},t^{s^{2}} \in \langle t,t^{s}\rangle$.
Doing the same thing with the relation $tststs=ststst$ one easily gets
\[
t^{s^{-1}}=t^{-1}\,(t^{s})^{-1}\,t^{s^{2}}\,t^{s}\,t \quad \text{and} \quad
t^{s^{3}}=t^{s^{2}}\,t^{s}\,t\, (t^{s})^{-1}\,(t^{s^{2}})^{-1}\,.
\]
This can be extended to any even Artin-type relation $(ts)^{k}=(st)^{k}$ giving
\begin{equation}\label{Eq2_1}
\begin{array}{c}
t^{s^{-1}} = t^{-1}\, (t^{s})^{-1} \cdots (t^{s^{k-2}})^{-1}\, t^{s^{k-1}}\, t^{s^{k-2}} \cdots t^{s}\, t\,,\\
t^{s^{k}} = t^{s^{k-1}} \cdots t^{s}\, t\, (t^{s})^{-1} \cdots (t^{s^{k-1}})^{-1}\,.
\end{array}
\end{equation}
As a consequence we see that $t^{s^i} \in \langle t, t^{s}, t^{s^{2}}, \dots, t^{s^{k-1}}\rangle$ for any integer $i$.


\section{Poly-freeness of even Artin groups of FC type}

In this section we prove that even Artin groups of FC type are poly-free (Theorem \ref{Thm3_17}). 
We begin with a characterization of these groups in terms of their defining graphs.

\begin{lem}\label{Lem3_1}
Let $A_\Gamma$ be an even Artin group. 
Then $A_{\Gamma}$ is of FC type if and only if every triangular subgraph of $\Gamma$ has at least two edges labelled with $2$.
\end{lem}

\begin{proof}
From  the  classification of finite Coxeter groups (see Bourbaki \cite{Bourb1}, for example) follows that an even Artin group $A_{\Gamma}$ is of spherical type if and only if $\Gamma$ is complete and for each vertex there is at most one edge with label greater than 2 involving it. 
Suppose that $A_{\Gamma}$ is of FC type. 
Let $\Omega$ be a triangular subgraph of $\Gamma$. 
Then $A_{\Omega}$ is even and of spherical type, hence, by the above, $\Omega$ has at least two edges labelled with $2$. 
Suppose now that every triangular subgraph of $\Gamma$ has at least two edges labelled with $2$. 
Let $\Omega$ be a complete subgraph of $\Gamma$. 
Then $\Omega$ is even and every triangular subgraph of $\Omega$ has at least two edges labelled with $2$, hence, by the above, $A_{\Omega}$ is of spherical type. 
So, $A_{\Gamma}$ is of FC type.
\end{proof}

The proof of Theorem \ref{Thm3_17} is based on the following.

\begin{prop}\label{Prop3_2}
Assume that $A_\Gamma$ is even and of FC type. 
Then there is a free group $F$ such that $A_\Gamma=F\rtimes A_1$, where $A_1$ is an even Artin group of FC type based on a proper subgraph of $\Gamma$.
\end{prop}

We proceed now with some notations needed for the proof of Proposition \ref{Prop3_2}. 
Fix some vertex $z$ of $\Gamma$.  
Recall that the \emph{link} of $z$ in $\Gamma$ is the full subgraph $\lk (z,\Gamma)$ of $\Gamma$ with vertex set $V(\lk(z,\Gamma))= \{s \in S \mid s \neq z \text{ and } m_{s,z} \neq \infty \}$.
As ever, we see $\lk(z,\Gamma)$ as a labelled graph, where the labels are the same as in the original graph $\Gamma$.
We set $L=\lk(z,\Gamma)$, and we denote by $\Gamma_1$ the full subgraph of $\Gamma$ spanned by $S \setminus \{z\}$.
We denote by $A_1$ and $A_L$ the subgroups of $A_\Gamma$ generated by $V(\Gamma_1)$ and $V(L)$, respectively. 
Recall from Section \ref{Sec2} that $A_1$ and $A_L$ are the Artin groups associated with the graphs $\Gamma_1$ and $L$, respectively (so this notation is consistent).

As pointed out in Section \ref{Sec2}, since $A_\Gamma$ is even, the inclusion map $A_1 \hookrightarrow A$ has a retraction $\pi_1: A_\Gamma \to A_1$ which sends $z$ to $1$ and sends $s$ to $s$ if $s \neq z$.
Similarly, the inclusion map $A_L \hookrightarrow A_1$ has a retraction $\pi_L : A_1 \to A_L$ which sends $s$ to $s$ if $s \in V(L)$, and sends $s$ to $1$ if $s \not\in V(L)$. 
It follows that $A_\Gamma$ and $A_1$ split as semi-direct products $A_\Gamma = \Ker (\pi_1) \rtimes A_1$ and $A_1= \Ker (\pi_L) \rtimes A_L$.

For $s \in V(L)$ we denote by $k_s$ the integer such that $m_{z,s}=2k_s$.
Lemma \ref{Lem3_1} implies the following statement.
This will help us to describe $A_L$ as an iterated HNN extension.

\begin{lem}\label{Lem3_3}
Let $s,t$ be two linked vertices of $L$.
\begin{itemize}
\item[(1)]
Either $k_s=1$, or $k_t=1$.
\item[(2)]
If $k_s >1$, then $m_{s,t}=2$.
\end{itemize}
\end{lem}

Let $L_1$ be the full subgraph of $L$ spanned by the vertices $s \in V(L)$ such that $k_s=1$.
Lemma \ref{Lem3_3} implies that $L\setminus L_1$ is totally disconnected. 
We set $V(L\setminus L_1)=\{x_1,\ldots,x_n\}$.
Again, from Lemma \ref{Lem3_3}, we deduce that, for each $i\in \{1, \dots,n\}$, if the vertex $x_i$ is linked to some vertex $s \in V(L_1)$, then the label of the edge between $x_i$ and $s$ must be $2$. 
For each $i \in \{1, \dots, n \}$ we set $S_i=\lk (x_i,L)$, and we denote by $X_i$ the full subgraph of $L$ spanned by $\{x_1,\ldots, x_i\} \cup V(L_1)$ and $X_0=L_1$.
Note that $S_i$ is a subgraph of $L_1$ and, therefore, is a subgraph of $X_i$. The subgraphs of $\Gamma$ that we have defined so far are sitting as follows inside $\Gamma$
$$S_i\subseteq L_1=X_0\subseteq X_1\subseteq\ldots\subseteq X_n= L\subseteq\Gamma_1\subseteq\Gamma$$
where $i\in\{1,\ldots,n\}$
The defining map of each of the HNN extensions will be the identity in the subgroup generated by the vertices commuting with $x_i$, that is, $\varphi_i= \Id :A_{S_i}\to A_{S_i}$. 
Then, writing down the associated presentation, we see that $A_{X_i}=(A_{X_{i-1}})*_{\varphi_i}$ with stable letter $x_i$. 
So, we get the following.

\begin{lem}\label{Lem3_4}
We have $A_L = ((A_{L_1} *_{\varphi_1} ) *_{\varphi_2} \cdots ) *_{\varphi_n}$.
\end{lem}

Now, fix a set $\NN_1$ of normal forms for $A_{L_1}$ (for example, the set of shortlex geodesic words with respect to some ordering in the standard generating system). 
We want to use Britton's lemma to obtain a set of normal forms for $A_L$ in terms of $\NN_1$.  
To do so, first, for each $i\in \{1, \dots,n\}$, we need to determine a set of representatives of the right cosets of $A_{S_i}$ in $A_{L_1}$. 
The natural way to do it is as follows. 
Consider the projection map $\pi_{S_i}: A_{L_1} \to A_{S_i}$ which sends $s\in V(L)$ to $s$ if $s \in V(S_i)$ and sends $s$ to $1$ otherwise. 
Observe that $A_{L_1}=A_{S_i}\ltimes\Ker(\pi_{S_i})$. 
Then, $\Ker(\pi_{S_i})$ is a well-defined set of representatives of the right cosets of $A_{S_i}$ in $A_{L_1}$.

In our next result we will use this set of representatives together with Britton's lemma to construct a set $\NN_L$ of normal forms  for $A_L$.
More precisely, $\NN_L$ denotes the set of words of the form
\[
w_0 x_{\alpha_1}^{\varepsilon_1} w_{1} \cdots x_{\alpha_m}^{\varepsilon_m} w_{m} \,,
\]
where $w_j \in \NN_1$ for all $j \in \{0,1, \dots, m\}$, $\alpha_j \in \{1, \dots, n\}$, $\bar w_j \in \Ker(\pi_{S_{\alpha_j}})$ and $\varepsilon_j \in \{\pm 1\}$ for all $j \in \{1, \dots, m\}$, and there is no subword of the form $x_{\alpha}^{\varepsilon} x_\alpha^{-\varepsilon}$ with $\alpha \in \{ 1, \dots, n\}$. 

\begin{lem}\label{Lem3_5}
The set $\NN_L$ is a set of normal forms for $A_{L}$.
\end{lem}

\begin{proof}
For $i \in \{0,1, \dots, n\}$, we denote by $\NN_{L,i}$ the set of words of the form
\begin{equation}\label{Eq3_1}
w_0 x_{\alpha_1}^{\varepsilon_1} w_{1} \cdots x_{\alpha_m}^{\varepsilon_m} w_{m} \,,
\end{equation}
where $w_j \in \NN_1$ for all $j \in \{0,1, \dots, m\}$, $\alpha_j \in \{1, \dots, i\}$, $\bar w_j \in \Ker(\pi_{S_{\alpha_j}})$ and $\varepsilon_j \in \{\pm 1\}$ for all $j \in \{1, \dots, m\}$, and there is no subword of the form $x_{\alpha}^{\varepsilon} x_\alpha^{-\varepsilon}$ with $\alpha \in \{ 1, \dots, i\}$.
We prove by induction on $i$ that $\NN_{L,i}$ is a set of normal forms for $A_{X_i}$.
Since $A_L=A_{X_n}$, this will prove the lemma.

The case $i=0$ is true by definition since $A_{L_1}=A_{X_0}$ and $\NN_{L,0}=\NN_1$.
So, we can assume that $i \ge 1$ plus the inductive hypothesis.
Recall that $A_{X_i}=(A_{X_{i-1}})*_{\varphi_i}$, where $\varphi_i$ is the identity map on $A_{S_i}$. 
By induction, $\NN_{L,i-1}$ is a set of normal forms for $A_{X_{i-1}}$. 
We want to apply Proposition \ref{Prop2_1}, so we also need a set $T_i$ of representatives of the right cosets of $A_{S_i}$ in $A_{X_{i-1}}$. 
Since $A_{L_1}=A_{S_i}\ltimes\Ker(\pi_{S_i})$, we see that we may take as $T_i$ the set of elements of $A_{X_{i-1}}$ whose normal forms, written as in Equation \ref{Eq3_1}, satisfy $\bar w_0 \in \Ker (\pi_{S_i})$. 
Now, take $g\in A_{X_i}$ and use Proposition \ref{Prop2_1} with the set $\NN_{L,i-1}$ of normal forms and the set $T_i$ of representatives to write a uniquely determined expression for $g$. 
The set of these expressions is clearly $\NN_{L,i}$.
\end{proof}

Given $g\in A_L$, we denote by $n(g)$ the normal form of $g$ in $\NN_L$.

The following sets will be crucial in our argument.	
The set $T_0$ is the set of $g\in A_L$ such that $n(g)$ does not begin with $s^{-1}$ or $s^{k_s}$ for any $s\in V(L)$.
In other words, $T_0$ denotes the set of elements in $A_L$ such that $n(g)$ has $w_0=1$, $\varepsilon_1=1$, and no consecutive sequence of $x_{\alpha_1}$'s of length $k_{x_{\alpha_1}}$ at the beginning. 
On the other hand, we set $T=T_0 \Ker(\pi_L)$.
The reason why these sets are so important to us is that $T$ will serve as index set for a free basis for the free group $F$ of Proposition \ref{Prop3_2}, that is, $F=F(B)$ is the free group with basis a set $B=\{b_g\mid g\in T\}$ in one-to-one correspondence with $T$.
We will also consider a smaller auxiliary free group $F_0=F(B_0)$ with basis $B_0=\{ b_h\mid h\in T_0\}$.
Observe that, since $T_0\subseteq A_L$, we have that any $g\in T$ can be written in a unique way as $g=hu$ with $h\in T_0$ and $u\in\Ker(\pi_L)$.

At this point, we can explain the basic idea of the proof of Proposition \ref{Prop3_2}.
We will define an action of $A_1$ on $F$ and form the corresponding semi-direct product. 
Then we will construct an explicit isomorphism between this semi-direct product and the original group $A_\Gamma$ that takes  $F$ onto $\Ker(\pi_1)$. 
More explicitly, this isomorphism will be an extension of the homomorphism $\varphi: F \to\Ker(\pi_1)$ which sends $b_g$ to $z^g$ for all $g \in T$.

We proceed now to define the action of $A_1$ on $F$. 
This action should mimic the obvious conjugation action of $A_1$ on $\Ker(\pi_1)$. 
As a first step, we start defining an action of $A_L$ on $F_0$. 
This will be later extended to the desired action of $A_1$ on $F$.

Given $h\in T_0$, we denote by $\supp (h)$ the set of $x_i\in V(L\setminus L_1)$ which appear in the normal form $n(h)$.
We will need the following technical result.

\begin{lem}\label{Lem3_6}
Let $s \in V(L)$ and $h\in T_0$. 
\begin{itemize}
\item[(1)]
If $s\in V(L_1)$ and $s \in V(S_i)$ for every $x_i \in \supp(h)$ (including the case $h=1$), then $hs \not \in T_0$ and $hs^{-1} \not \in T_0$, but $shs^{-1}, s^{-1} h s \in T_0$.
\item[(2)]
If $s=x \in V(L \setminus L_1)$ and $h=x^{k_x-1}$, then $hs \not\in T_0$ and $hs^{-1} \in T_0$.
\item[(3)]
If $s=x \in V(L \setminus L_1)$ and $h=1$, then $hs \in T_0$ and $hs^{-1} \not\in T_0$.
\item[(4)]
We have $hs, hs^{-1} \in T_0$ in all the other cases.
\end{itemize}
\end{lem}

\begin{proof}
Since $h\in T_0$, $h$ has a normal form
\[
x_{\alpha_1}^{\varepsilon_1} w_{1} \cdots x_{\alpha_m}^{\varepsilon_m} w_m\,,
\]
with each $w_j$ the normal form of an element in $\Ker (\pi_{S_{\alpha_j}})$. 
In the case when $s=x \in V(L\setminus L_1)$, multiplying this expression by $x$ on the right, we get, after a possible cancellation $x^{-1}x$, a normal form.
So, $hx\not\in T_0$ if and only if $h=x^{k_x-1}$.
Similarly, we have $hx^{-1}\not\in T_0$ if and only if $h=1$.

Now, suppose that $s\in V(L_1)$.  
Assume that $s\in S_{\alpha_j}$ for $j=j_0+1,\dots,m$, but $s\not\in S_{\alpha_{j_0}}$. 
Then the normal form for $hs$ is
\[
x_{\alpha_1}^{\varepsilon_1} w_{1} \cdots w_{j_0-1} x_{\alpha_{j_0}}^{\varepsilon_{j_0}} w_{j_0}' \cdots x_{\alpha_{m-1}}^{\varepsilon_{m-1}} w_{m-1}' x_{\alpha_m}^{\varepsilon_m} w_m'\,,
\]
where $w_{j_0}'$ is the normal form for $\bar w_{j_0} s$, and, for $j \in \{j_0+1, \dots, m\}$, $w_j'$ is the normal form for $s^{-1} \bar w_j s$.
Hence, $hs\in T_0$. 
Similarly, $hs^{-1} \in T_0$.
On the contrary, assume that $s \in S_{\alpha_j}$ for all those $x_j$ appearing in the normal form for $h$.
Then the normal form for $hs$ is 
\[
s x_{\alpha_1}^{\varepsilon_1} w_1' \cdots x_{\alpha_m}^{\varepsilon_m} w_m'\,,
\]
where $w_j'$ is the normal form for $s^{-1} \bar w_j s$ for all $j \in \{1, \dots, m\}$.
This form begins with $s$, hence $hs\not\in T_0$.
However, the normal form for $s^{-1} h s$ is obtained from the above form by removing the $s$ at the beginning, hence $s^{-1} h s \in T_0$.
Similarly, $h s^{-1} \not\in T_0$ and $s h s^{-1} \in T_0$.
\end{proof}

Hidden in the proof of Lemma \ref{Lem3_6} is the fact that, if $g_1,g_2 \in A_{L_1}$ and $h \in T_0$, then $g_1 h g_2 \in A_{L_1} T_0$.
Moreover, by Lemma \ref{Lem3_5}, every element of $A_{L_1} T_0$ is uniquely written in the form $gh$ with $g \in A_{L_1}$ and $h \in T_0$.
In this case we set $u(gh) = h$.
So, by Lemma \ref{Lem3_6}, if $s \in V(L_1)$ and $h \in T_0$, then $u(hs)= s^{-1} h s$ if $s \in V(S_i)$ for every $x_i \in \supp(h)$, and $u(hs) = hs$ otherwise.
If $s=x \in V(L \setminus L_1)$, then $u(hx)$ is not defined if $h = x^{k_x-1}$, and $u(hx) = hx$ otherwise.

We turn now to define the action of $A_L$ on $F_0$.
We start with the action of the generators.
Let $s \in V(L)$.
For $h \in T_0$ we set
\[
b_h * s = \left\{
\begin{array}{ll}
b_{u(hs)} & \text{if } hs \in A_{L_1} T_0\,,\\
b_{x^{k_x-1}} \cdots b_{x} \, b_1 \, b_x^{-1} \cdots b_{x^{k_x-1}}^{-1} &\text{if } s=x \in V(L \setminus L_1) \text{ and } h=x^{k_x-1}\,.
\end{array} \right.
\]
Then we extend the map $B \to F_0$, $b_h \mapsto b_h*s$, to a homomorphism $F_0 \to F_0$, $f \mapsto f *s$.

\begin{lem}\label{Lem3_7}
The above defined homomorphism $*s : F_0 \to F_0$ is an automorphism.
\end{lem}

\begin{proof}
The result will follow if we show that the map has an inverse. 
Our candidate to inverse will be the map $*s^{-1}: F_0 \to F_0$ defined by
\[
b_h * s^{-1} = \left\{
\begin{array}{ll}
b_{u(hs^{-1})} & \text{if } hs^{-1} \in A_{L_1} T_0\,,\\
b_1^{-1}\, b_x^{-1} \cdots b_{x^{k_x-2}}^{-1}\, b_{x^{k_x-1}}\, b_{x^{k_x-2}} \cdots b_{x}\, b_1 &\text{if } s=x \in V(L \setminus L_1)\\
& \text{ and } h=1\,.
\end{array} \right.
\]

Assume first that either $s \in V(L_1)$, or $h \not\in \{ 1,x^{k_x-1}\}$ with $s=x \in V(L \setminus L_1)$. 
In this case we only have to check that $u(u(hs)s^{-1})=h=u(u(hs^{-1})s)$.
Observe that the condition $s \in V(L_1)$ and $s \in S_i$ for every $x_i\in\supp (h)$ is equivalent to $s\in V(L_1)$ and  $s \in S_i$ for every $x_i\in\supp(u(hs))$, thus, if that condition holds, we have $u(hs)=s^{-1}hs$ and $u(u(hs)s^{-1}) = u((s^{-1}hs)s^{-1}) = ss^{-1}hss^{-1} = h$.
If the condition fails, then $u(hs)=hs$ and $u(u(hs)s^{-1})=u((hs)s^{-1})=hss^{-1}=h$.
Analogously, one checks that $h=u(u(hs^{-1})s)$.

Now, we assume that $ s=x\in V(L\setminus L_1)$ and either $h=1$ or $h=x^{k_x-1}$. 
If $h=x^{k_x-1}$, then 
\begin{gather*}
(b_h*x)*x^{-1} =
(b_{x^{k_x-1}} \cdots b_{x}\, b_1\, b_x^{-1} \cdots b_{x^{k_x-1}}^{-1})*{x^{-1}}\\ =
b_{x^{k_x-2}} \cdots b_1\, (b_1*{x^{-1}})\, b_1^{-1} \cdots b_{x^{k_x-2}}^{-1} \\=
b_{x^{k_x-2}} \cdots b_1\, b_1^{-1}\, b_x^{-1} \cdots b_{x^{k_x-2}}^{-1}\, b_{x^{k_x-1}}\, b_{x^{k_x-2}} \cdots b_{x}\, b_1\, b_1^{-1} \cdots b_{x^{k_x-2}}^{-1}
= b_{x^{k_x-1}} = b_h\,.
\end{gather*}
On the other hand,
\[
(b_h*{x^{-1}})*x = b_{x^{k_x-2}}*x = b_{x^{k_x-1}} = b_h\,.
\]
The case when $h=1$ is analogous.
\end{proof}

To show that this action, defined just for the generators of $A_L$, yields an action of $A_L$ on $F_0$, we need to check that it preserves the Artin relations. 
We do it in the next two lemmas.

\begin{lem}\label{Lem3_8}
Let $s,t \in V(L)$ such that $m_{s,t}=2$.
Then $(g*s)*t = (g*t)*s$ for all $g \in F_0$.
\end{lem}

\begin{proof}
Since $s$ and $t$ are linked in $V(L)$, Lemma \ref{Lem3_1} implies that at least one of the vertices $s,t$ lies in $V(L_1)$. 
Without loss of generality we may assume $s\in V(L_1)$, i.e., $k_s=1$. 
Let $h \in T_0$.
Then $b_h*s=b_{u(hs)}$.

Assume first that either $ht \in T_0$, or $t\in S_i$ for every $x_i\in\supp(h)$. 
In this last case, we have $t \in S_i$ for any $x_i \in \supp(u(hs))$. 
Therefore
\[
(b_h*s)*t=b_{u(hs)}*t=b_{u(u(hs)t)} \quad \text{and}\quad (b_h*t)*s=b_{u(ht)}*s = b_{u(u(ht)s)}\,.
\]
Depending on whether $hs$ lies in $T_0$ or not we have $u(hs)=hs$ or $u(hs)=s^{-1}hs$, and the same for $t$. 
So, we have four cases to consider.
If $hs,ht \in T_0$, then $hst=hts \in T_0$ and
\[
u(u(hs)t) = u((hs)t)=hst=hts=u((ht)s)=u(u(ht)s)\,.
\]
If $hs \in T_0$ and $ht \not \in T_0$, then $hst\not\in T_0$ but $t^{-1}hts\in T_0$, hence
\[
u(u(hs)t) =u((hs)t)= t^{-1}hst = t^{-1}hts = u((t^{-1}ht)s)=u(u(ht)s)\,.
\]
Similarly, if $hs \not\in T_0$ and $ht \in T_0$, then $u(u(hs)t) = u(u(ht)s)$. 
If $hs,ht\not\in T_0$, then $s^{-1}hst\not\in T_0$ and $t^{-1}hts \not\in T_0$, thus 
\[
u(u(hs)t) = u((s^{-1}hs)t) = t^{-1}s^{-1}hst = s^{-1}t^{-1}hts = u((t^{-1}ht)s) = u(u(ht)s)\,.
\]

We are left with the case where $t=y \in V(L\setminus L_1)$ and $h=y^{k_y-1}$. 
Then, since $s$ and $y$ are linked, for every $\alpha\in\{0,1,\dots ,k_y-1\}$ we have $u(y^\alpha s) = s^{-1}y^\alpha s = y^\alpha$, thus
\[
(b_h*s)*y = b_h*y = b_{y^{k_y-1}} \cdots b_y\, b_1\, b_y^{-1} \cdots b_{y^{k_y-1}}^{-1}\,,
\] \[
(b_h*y)*s=(b_{y^{k_y-1}} \cdots b_y\,b_1\,b_y^{-1} \cdots b_{y^{k_y-1}}^{-1})*s = b_{y^{k_y-1}} \cdots b_y\, b_1\, b_y^{-1} \cdots b_{y^{k_y-1}}^{-1}\,.
\proved
\]
\end{proof}

If $w$ is a word over $\{s,t\}$, where $s,t \in V(L)$, and if $h \in T_0$, we define $b_h * w$ by induction on the length of $w$ by setting $b_h*1=b_h$ and $b_h*(ws) = (b_h*w)*s$.

\begin{lem}\label{Lem3_9}
Let $s,t \in V(L)$ such that $m_{s,t} = 2k >2$.
Then $b_h*((st)^k)=b_h*((ts)^k)$ for all $h \in T_0$.
\end{lem}

\begin{proof}
Note that, since $s$ and $t$ are linked and the edge between them is labelled with $2k>2$, Lemma \ref{Lem3_1} implies that 
$s,t\in V(L_1)$. 
Take $h\in T_0$.
Then, in a similar way as in the proof of Lemma \ref{Lem3_8}, we have $hs\in T_0$ if and only if $u(hstst \cdots t)s\in T_0$ and this is also equivalent to $u(htsts \cdots t)s\in T_0$. 
The same thing happens for $t$. 
So, we may distinguish essentially the same cases as in the first part of the proof of Lemma \ref{Lem3_8} and get the following.
If $hs,ht\in T_0$, then
\[ 
(b_h)*{(st)^k}=b_{h(st)^k} = b_{h(ts)^k} = (b_h)*{(ts)^k}\,.
\]
If $hs \in T_0$ and $ht \not\in T_0$, then
\[ 
(b_h)*{(st)^k} = b_{t^{-k}h(st)^k} = b_{t^{-k}h(ts)^k} = (b_h)*{(ts)^k}\,.
\]
Similarly, if $hs\not\in T_0$ and $ht\in T_0$, then $(b_h)*{(st)^k} = (b_h)*{(ts)^k}$. 
If $hs, ht\not\in T_0$, then  
\[
(b_h)*{(st)^k}=b_{(st)^{-k}h(st)^k} = b_{(ts)^{-k}h(ts)^k}=(b_h)*{(ts)^k}\,.
\proved
\]
\end{proof}

All the previous discussion implies the following.

\begin{lem}\label{Lem3_10}
The mappings $*s$, $s \in V(L)$, yield a well-defined right-action $F_0 \times A_L \to F_0$, $(u,g) \mapsto u*g$.
\end{lem}

Moreover, this action behaves as one might expect. 
To show this, we will need the following technical lemma.

\begin{lem}\label{Lem3_11}
Let $g\in A_{L_1}T_0$, and let $n(g) = w_0 x_{\alpha_1}^{\varepsilon_1} w_1 \cdots x_{\alpha_n}^{\varepsilon_n} w_n$ be its normal form. 
Then any prefix of $n(g)$ also represents an element in $A_{L_1}T_0$.
\end{lem}

\begin{proof}
The only case where it is not obvious is when the prefix is of the form $w_0 x_{\alpha_1}^{\varepsilon_1} w_1 \cdots x_{\alpha_j}^{\varepsilon_j} u_j$ with $u_j$ a prefix of $w_j$.
Let $h$ be the element of $A_L$ represented by $w_0 x_{\alpha_1}^{\varepsilon_1} w_1 \cdots x_{\alpha_j}^{\varepsilon_j} u_j$, and let $h'$ be the element represented by $w_0 x_{\alpha_1}^{\varepsilon_1} w_1 \cdots x_{\alpha_j}^{\varepsilon_j}$.
It is clear that $h' \in A_{L_1} T_0$.
Moreover, $h=h' \bar u_j$, hence, as pointed out after the proof of Lemma \ref{Lem3_6}, we have $h \in A_{L_1} T_0$.
\end{proof}

\begin{lem}\label{Lem3_12}
For every $g\in A_{L_1}T_0$ we have $b_1 *g=b_{u(g)}$.
\end{lem}

\begin{proof}
Set again $n(g) = w_0 x_{\alpha_1}^{\varepsilon_1} w_1 \cdots x_{\alpha_n}^{\varepsilon_n} w_n$. 
Let $vs^\varepsilon$ be a prefix of $n(g)$ with $s$ a vertex and $\varepsilon \in \{ \pm 1 \}$.
Note that Lemma \ref{Lem3_11} implies that $vs^\varepsilon$ and $v$ represent elements in $A_{L_1}T_0$. 
We are going to prove that, if $b_1*\bar v = b_{u(\bar v)}$, then also $b_1 * {\bar vs^\varepsilon}= b_{u( \bar vs^\varepsilon)}$.
Note that this will imply the result. 
As $\bar vs^\varepsilon$ lies in $A_{L_1}T_0$, the element $u(\bar vs^\varepsilon)$ is well-defined.
Observe that $u(u(\bar v)s^\varepsilon)$ is also well-defined. 
We have $b_1 * (\bar vs^\varepsilon) = b_{u(\bar v)}*s^\varepsilon = b_{u(u(\bar v)s^\varepsilon)}$.
So, we only need to show that $u (u (\bar v) s^\varepsilon) = u(\bar v s^\varepsilon)$. 
Set $\bar v = q h$ with $h \in T_0$ and $q \in A_{L_1}$. 
Then $u(\bar v) = h$, thus, by Lemma \ref{Lem3_6}, $u (u (\bar v)s^\varepsilon) = u (hs^\varepsilon) = u(q h s^\varepsilon) = u( \bar v s^\varepsilon)$.
\end{proof}

Our next objective is to extend the action  of $A_L$ on $F_0$ to an action of $A_1$ on $F$. 
Recall that $T=T_0\Ker(\pi_L)$ and that any $h\in T$ can be written in a unique way as $h=h_0 u$ with $h_0\in T_0$ and $u \in \Ker(\pi_L)$. 
Taking this into account we set $b_h = b_{h_0u} = b_{h_0} \cdot u$.
We extend this notation to any element $\omega = \prod b_{h_i}^{\varepsilon_i} \in F_0$ by setting $\omega \cdot u=\prod b_{h_iu}^{\varepsilon_i}$.

Now, let $g\in A_1$ and $h\in T$. 
We write $h=h_0 u$ with $h_0 \in T_0$ and $u\in\Ker(\pi_L)$. 
So, with the previous notation, we have $b_h=b_{h_0} \cdot u$. 
Then we set
\[
b_h*g = (b_{h_0}*{\pi_L(g)}) \cdot (\pi_L(g)^{-1}ug)\,.
\]
We can also write this action as follows. 
Let $\omega = \prod b_{h_i}^{\varepsilon_i} \in F_0$ and let $u\in\Ker(\pi_L)$. 
Then 
\begin{equation}\label{Eq3_2}
(\omega \cdot u)*g=(\omega * \pi_L(g)) \cdot (\pi_L(g)^{-1}ug)\,.
\end{equation}

\begin{lem}\label{Lem3_13}
The above defined map $F \times A_1 \to F$, $(\omega, g) \mapsto \omega * g$, is a well-defined right-action of $A_1$ on $F$.
\end{lem}

\begin{proof}
The lemma is essentially a consequence of the fact that the action of $A_L$ on $F_0$ is well-defined. 
Let $g_1,g_2\in A_1$ and let $b_h=b_{h_0} \cdot u\in B$, where $h_0 \in T_0$ and $u \in \Ker(\pi_L)$. 
Then, using Equation \ref{Eq3_2}, 
\begin{gather*}
(b_h*{g_1})*{g_2} = 
((b_{h_0}\cdot u)*{g_1})*{g_2} =
\big( (b_{h_0} * \pi_L(g_1)) \cdot (\pi_L(g_1)^{-1} u g_1) \big)*{g_2}\\ =
((b_{h_0}*{\pi_L(g_1)})*{\pi_L(g_2)}) \cdot (\pi_L(g_2)^{-1} (\pi_L(g_1)^{-1} u g_1 ) g_2)
\end{gather*}
\[
 =
(b_{h_0}*{\pi_L(g_1g_2)}) \cdot (\pi_L(g_1g_2)^{-1} u g_1g_2)= 
b_h*{g_1g_2}\,.
\proved
\]
\end{proof}

Recall the homomorphism $\varphi : F \to \Ker (\pi_1)$ that sends $b_h$ to $z^h$ for all $h \in T$.
We consider the semi-direct product $G = A_1 \ltimes F$ associated with the above action, and we turn to define an extension of $\varphi$ to $G$.

\begin{lem}\label{Lem3_14}
The map $G \to A$, $(g, \omega) \mapsto g\,\varphi(\omega)$, is a well-defined homomorphism.
\end{lem}

\begin{proof}
We have to check that, for all $g_1$, $g_2\in A_1$ and all $\omega_1$, $\omega_2\in F$, we have
\begin{gather*}
\varphi (g_1,\omega_1) \varphi(g_2,\omega_2) = g_1g_2 \varphi(\omega_1)^{g_2} \varphi(\omega_2)\\
 = \varphi (g_1g_2,(\omega_1*{g_2})\omega_2) = g_1 g_2 \varphi( (\omega_1*{g_2}) \omega_2)\,.
\end{gather*}
Since the restriction of $\varphi$ to $F$ is a  group homomorphism, this is equivalent to show that $\varphi (\omega_1)^{g_2} = \varphi(\omega_1*{g_2})$.
It is enough to prove this for the group generators. 
So, we can assume that $\omega_1=b_h$ for some $h\in T$, and that $g_2=s$ for some vertex $s\in V(\Gamma)$, $s\neq z$. 
We have $\varphi(b_h)^s = s^{-1} z^h s = s^{-1} h^{-1} z h s$, and we need to check that this is equal to $\varphi(b_h*s)$. 
To see it we set $h=h_0 u$, with $h_0 \in T_0$ and $u \in \Ker(\pi_L)$, so that $b_h = b_{h_0} \cdot u$.

Observe first that, if $s\not\in V(L)$, then $\pi_L(s)=1$, thus $b_h*s=b_{h_0} \cdot (us) =b_{h_0 u s} = b_{hs}$, and therefore 
\[
\varphi(b_h*s)=\varphi (b_{hs})= z^{hs} = \varphi (b_h)^s\,.
\]

So, from now on, we will assume that $s\in V(L)$. 
Then $\pi_L(s)=s$,  thus $b_h^s = (b_{h_0}*s) \cdot (s^{-1}us)$, and therefore
\[
\varphi(b_h*s) = \varphi(b_{h_0}*s)^{s^{-1}us} = s^{-1} \varphi(b_{h_0}*s)^{s^{-1}u}s\,.
\]
So, we only have to prove that $\varphi(b_{h_0}*s) =z^{h_0s}$.
We  distinguish three different cases.
If $h_0 s \in T_0$, then $b_{h_0}*s = b_{h_0s}$, thus
\[
\varphi (b_{h_0}*s) = \varphi (b_{h_0 s})=z^{h_0 s}\,.
\]
If $h_0s \not \in T_0$, $s \in V(L_1)$ and $s\in S_i$ for every $x_i\in\supp(h_0)$, then $b_{h_0}*s = b_{s^{-1} h_0 s}$, thus
\[ 
\varphi(b_{h_0}*s) = \varphi( b_{s^{-1} h_0 s})= z^{s^{-1} h_0 s}=z^{h_0 s}\,.
\]
Finally, if $s=x\in V(L\setminus L_1)$ and $h_0=x^{k_x-1}$, then 
\[
b_{h_0}*x = b_{x^{k_x-1}} \cdots b_{x}\,b_1\,b_x^{-1} \cdots b_{x^{k_x-1}}^{-1}\,,
\]
thus  
\[ 
\varphi (b_{h_0}*x) = 
\varphi(b_{x^{k_x-1}} \cdots b_{x}\, b_1\, b_x^{-1} \cdots b_{x^{k_x-1}}^{-1})
\] \[
= z^{x^{k_x-1}} \cdots z^x\, z\, (z^x)^{-1} \cdots (z^{x^{k_x-1}})^{-1} = 
z^{x^{k_x}} = z^{h_0x}\,.
\proved
\]
\end{proof}

Now, we want to define the inverse map of $\varphi$. 
We do it by giving the images of the Artin generators of $A$, that is, the vertices of $\Gamma$.

\begin{lem}\label{Lem3_15}
There is a well-defined homomorphism $\psi : A \to G$ that sends $s$ to $s$ for all $s \in V(\Gamma) \setminus \{z\}$, and sends $z$ to $b_1$.
\end{lem}

\begin{proof}
We have to check that the Artin relations are preserved by $\psi$. 
Note that it suffices to check it for the Artin relations that involve $z$ and some $s \in V(\Gamma) \setminus \{ z\}$. 
If $s\not\in V(L)$, then there is nothing to check because, in that case, $s$ and $z$ are not linked in $\Gamma$, hence there is no relation between them. 
If $s\in V(L)$, then we can rewrite the Artin relation as
\[
z^{s^{k_s}} = z^{s^{k_s-1}} \cdots z^s\, z\, (z^s)^{-1} \cdots (z^{s^{k_s-1}})^{-1}\,.
\]
We include here the case $s \in V(L_1)$, where we have $k_s=1$ and the above formula is $z^s=z$.
Applying $\psi$ to the left hand side of this equation we get
\[
\psi(z^{s^{k_s}}) = s^{-k_s}\, b_1\, s^{k_s} = s^{-1}\, b_{s^{k_s-1}}\, s = b_{s^{k_s-1}} \cdots b_{s}\ b_1\, b_s^{-1} \cdots b_{s^{k_s-1}}^{-1}\,,
\]
which is exactly what we get applying $\psi$ to the right hand side.
\end{proof}

Lemma \ref{Lem3_14} and Lemma \ref{Lem3_15} show part of the following result.

\begin{prop}\label{Prop3_16}
The maps $\varphi:G\to A$ and $\psi:A\to G$ are well-defined group isomorphisms.
\end{prop}

\begin{proof}
We have already seen that both maps are group homomorphisms. 
We claim that they are inverses of each other. 
This will prove the result. 
Let $s \in V(\Gamma)$, $s \neq z$. 
We have $(\varphi \circ \psi)(s) = \varphi(s)=s$ and $(\psi \circ \varphi)(s) = \psi(s) =s$. 
Also $(\varphi \circ \psi)(z)=\varphi(b_1)=z$ and $(\psi \circ \varphi)(b_1) = \psi(z)=b_1$.
Moreover, Lemma \ref{Lem3_12} implies that $b_1$ and the Artin generators of $A_1$ generate the whole group $G$, so $\psi\circ \varphi$ is the identity of $G$.
Similarly, $\varphi \circ \psi$ is the identity of $A$.
\end{proof}

Now, we obtain immediately our main result.

\begin{thm}\label{Thm3_17}
Every even Artin group of FC type is poly-free.
\end{thm}

\begin{proof}
By Proposition \ref{Prop3_16}, $A_{\Gamma} \simeq G = F \rtimes A_{1}$. 
By induction we may assume that $A_{1}$ is poly-free, thus $A$ is also poly-free.
\end{proof}


\section{Residually finiteness}\label{Sec4}

In this section we will show that even Artin groups of FC type are residually finite. 
Recall that a group $G$ is said to be \emph{residually finite} if, for every $g\in G\setminus \{1\}$, there is a normal subgroup of finite index in $G$ not containing $g$. 
It is well-known that being residually finite is not closed under short exact sequences, in the sense that, if $N$ is a normal subgroup of $G$ and both $N$ and $G/N$ are residually finite, then one cannot deduce the same for $G$ itself. 
However, the situation changes if we work under some extra hypothesis.
For example, a direct product of residually finite groups is residually finite. 
This can be generalized to the following result of Boler--Evans \cite{BolEva1} that will be crucial in our argument.

\begin{thm}[Boler--Evans \cite{BolEva1}]\label{Thm4_1}
Let $G_1, G_2$ be residually finite groups, and let $L\leq G_1,G_2$ such that both split as $G_i=H_i\rtimes L$. 
Then the amalgamated free product $G=G_1\ast_{L} G_2$ is residually finite.
\end{thm}

Using this we get the following.

\begin{thm}\label{Thm4_2}
Every even Artin group of FC type is residually finite.
\end{thm}

\begin{proof}
We argue by induction on the number of vertices of $\Gamma$.
Assume first that $\Gamma$ is complete. 
Then, by definition, $A_\Gamma$ is of spherical type, hence, by Cohen--Wales \cite{CohWal1} and Digne \cite{Digne1}, $A_{\Gamma}$ is linear, and therefore $A_{\Gamma}$ is residually finite.

Assume that $\Gamma$ is not complete.
Then we can choose two distinct vertices $s,t \in S=V(\Gamma)$ such that $m_{s,t} = \infty$.
Set $X = S \setminus \{ s \}$, $Y = S \setminus \{t\}$ and $Z = S \setminus \{s,t\}$.
From the presentation of $A$ follows that $A=A_X*_{A_Z} A_Y$.
Moreover, since $A$ is even, the inclusion map $A_Z \hookrightarrow A_X$ has a retraction $\pi_{X,Z}:A_X \to A_Z$ which sends $r$ to $r$ for all $r \in Z$ and sends $t$ to $1$, hence $A_X=\Ker(\pi_{X,Z}) \rtimes A_Z$.
Similarly, the inclusion map $A_Z \hookrightarrow A_Y$ has a retraction $\pi_{Y,Z}: A_Y \to A_Z$, hence $A_Y = \Ker (\pi_{Y,Z}) \rtimes A_Z$.
By the inductive hypothesis, $A_X$ and $A_Y$ are residually finite, hence, by Theorem \ref{Thm4_1}, $A$ is also residually finite. 
\end{proof}



\end{document}